\documentclass[pdflatex,sn-mathphys-num]{sn-jnl}

\usepackage{graphicx}%
\usepackage{multirow}%
\usepackage{amsmath,amssymb,amsfonts}%
\usepackage{amsthm}%
\usepackage{mathrsfs}%
\usepackage[title]{appendix}%
\usepackage{xcolor}%
\usepackage{textcomp}%
\usepackage{manyfoot}%
\usepackage{booktabs}%
\usepackage{algorithm}%
\usepackage{algorithmicx}%
\usepackage{algpseudocode}%
\usepackage{listings}%
\usepackage[utf8]{inputenc}

\theoremstyle{thmstyleone}%
\newtheorem{theorem}{Theorem}
\newtheorem{proposition}{Proposition}%

\theoremstyle{thmstyletwo}%
\newtheorem{example}{Example}%

\theoremstyle{thmstylethree}%
\newtheorem{definition}{Definition}%
\newtheorem{Remark}{Remark}%
\newtheorem{lemma}{Lemma}%
\begin{document}

\title[Article Title]{\Large On the equivalent Zbǎganu  constant associated  with isosceles orthogonality in Banach spaces}

\author[1]{\fnm{Tingting} \sur{Hu}}\email{y25060048@stu.aqnu.edu.cn}
\equalcont{These authors contributed equally to this work.}

\author*[1,2]{\fnm{Qi} \sur{Liu}}\email{liuq67@aqnu.edu.cn}
\equalcont{These authors contributed equally to this work.}

\author[1]{\fnm{Mengmeng} \sur{Bao}}\email{060921003@stu.aqnu.edu.cn}
\equalcont{These authors contributed equally to this work.}

\author[1]{\fnm{Yuxin} \sur{Wang}}\email{y24060028@stu.aqnu.edu.cn}
\equalcont{These authors contributed equally to this work.}

\affil [1]{\orgdiv{School of Mathematics and physics},  \orgaddress{\street{Anqing Normal University}, \city{Anqing}, \postcode{246133}, \country{China}}}

\affil[2]{\orgdiv{International Joint Research Center of Simulation and Control for Population Ecology of Yantze River in Anhuil },\orgaddress{Anqing Normal University},\city{Anqing}, \postcode{246133}, \country{China}}


\abstract{This paper systematically investigates a new geometric constant  associated with isosceles orthogonality in Banach spaces. By establishing the connection between this new constant and a classical function, sharp upper and lower bounds for the constant are derived. Specifically, the space is exactly a Hilbert space when the new constant reaches its lower bound; in finite-dimensional spaces, if the new constant attains its upper bound, it implies that the space does not possess uniform non-squareness. }

\keywords{Banach space, Geometric constants, Zbăganu constant, Uniformly smooth}



\maketitle

\section{Introduction}
The study of the geometric properties of Banach spaces constitutes an important field within the discipline of functional analysis. In this theory, geometric constants that quantify the geometric characteristics of normed spaces play a significant role. For example, Clarkson introduced the concept of the modulus of convexity to describe uniformly convex spaces, while the von Neumann-Jordan constant is used to characterize the geometric properties of Hilbert spaces and other such spaces. 

In Euclidean geometry, orthogonality is a fundamental concept. It is not only reflected in the fourth axiom of Euclidean geometry but also plays a pivotal role in the Pythagorean theorem. However, Banach space geometry differs significantly from Euclidean geometry, as it lacks an intuitive notion of orthogonality. With the advancement of Banach space geometry, various generalized forms of orthogonality have been introduced into normed linear spaces. For example, James \cite{00} proposed isosceles orthogonality ($\perp_I$) and Pythagorean orthogonality ($\perp_P$):
\[
x_1 \perp_I x_2 \quad \text{if and only if} \quad \|x_1+x_2\|=\|x_1-x_2\|
\]
and
\[
x_1 \perp_P x_2 \quad \text{if and only if} \quad \|x_1-x_2\|^2=\|x_1\|^2+\|x_2\|^2
\]

Throughout this paper, we assume that \( X \) is a real Banach space with dimension \(\dim X \geq 2\). The unit ball and unit sphere of \( X \) are denoted by \( B_X = \{x \in X : \|x\| \leq 1\} \) and \( S_X = \{x \in X : \|x\| = 1\} \), respectively. Below, we review some fundamental concepts of geometric properties in Banach spaces and introduce several well-known geometric constants relevant to our study.

A Banach space \( X \) is termed uniformly non-square if there exists \( \delta \in (0, 1) \) such that for any \( x_1, x_2 \in S_X \), either \( \|x_1 + x_2\| \leq 2(1 - \delta) \) or \( \|x_1 - x_2\| \leq 2(1 - \delta) \). Conversely, if for every \( \varepsilon > 0 \), there exist \( x_1, x_2 \in S_X \) satisfying \( \|x_1 \pm x_2\| > 2 - \varepsilon \), then \( X \)  is not uniformly non-square.

A Banach space \( X \) is defined to have (weak) normal structure \cite{001} when, for any (weakly compact) closed, bounded, and convex subset \( K \) of \( X \) that contains at least two distinct points, there is a point \( x_0 \in K \) satisfying  

\[\sup\{||x_0 - x_2|| : x_2 \in K\} < d(K) = \sup\{||x_1 - x_2|| : x_1, x_2 \in K\}.\]  

Moreover, \( X \) is said to have uniform normal structure if there exists a constant \( 0 < c < 1 \) such that, for every closed, bounded, and convex subset \( K \) of \( X \) with more than one point, there exists \( x_0 \in K \) for which  

\[\sup\{||x_0 - x_2|| : x_2 \in K\} < c d(K) = c \sup\{||x_1 - x_2|| : x_1, x_2 \in K\}.\]  

Normal and weakly normal structures are fundamental in fixed point theory. Notably, in reflexive Banach spaces, normal structure and weakly normal structure are equivalent. James \cite{003} proved that uniform non-squareness in a Banach space \( X \) implies reflexivity. Kirk \cite{005} further showed that a reflexive Banach space \( X \) with normal structure necessarily has the fixed point property. Additionally, it has been established in \cite{006} that every uniformly non-square Banach space enjoys the fixed point property.

A Banach space \( X \) is uniformly convex if for every \( 0 < \varepsilon \leq 2 \), there exists \( \delta > 0 \) such that for any \( x_1, x_2 \in S_X \) with \( \|x_1 - x_2\| \geq \varepsilon \), the inequality \( \|x_1 + x_2\|/2 \leq 1 - \delta \) holds. The space \( X \) is strictly convex if for any distinct \( x_1, x_2 \in S_X \), \( \|x_1 + x_2\| < 2 \).  

The von Neumann-Jordan constant \( C_{\text{NJ}}(X) \), introduced by Clarkson \cite{002}, is defined as:  
\[
C_{\text{NJ}}(X) = \sup\left\{\frac{\|x_1 + x_2\|^2 + \|x_1 - x_2\|^2}{2(\|x_1\|^2 + \|x_2\|^2)} : x_1, x_2 \in X, (x_1, x_2) \neq (0, 0)\right\}.
\]  
Key properties include:\\  
(i) \( 1 \leq C_{\text{NJ}}(X) \leq 2 \). \\ 
(ii) \( X \) is a Hilbert space if and only if \( C_{\text{NJ}}(X) = 1 \).\\  
(iii) \( X \) is uniformly non-square if and only if \( C_{\text{NJ}}(X) < 2 \) \cite{004}. \\ 

The modified von Neumann-Jordan constant \cite{002} is given by:  
\[
C'_{\text{NJ}}(X) = \sup\left\{\frac{\|x_1 + x_2\|^2 + \|x_1 - x_2\|^2}{4} : x_1, x_2 \in S_X\right\}.
\]  
Papini \cite{008} introduced another variant:  
\[
C''_{\text{NJ}}(X) = \sup\left\{\frac{\|x_1 + x_2\|^2 + \|x_1 - x_2\|^2}{2(\|x_1\|^2 + \|x_2\|^2)} : x_1, x_2 \in X, (x_1, x_2) \neq (0, 0), x_1 \perp_I x_2\right\},
\]  
and showed that it retains properties (ii) and (iii) above.  

The modulus of smoothness \( \rho(t) \) \cite{009} is defined as:  
\[
\rho(t) = \sup\left\{\frac{\|x_1 + tx_2\| + \|x_1 - tx_2\|}{2} - 1 : x_1, x_2 \in S_X\right\},
\]  
where $t \geq 0$.
A Banach space \( X \) is uniformly smooth if \( \lim_{t \to 0} \frac{\rho(t)}{t} = 0 \).  

The modulus of convexity \( \delta_X(\varepsilon) \) \cite{010} is defined as:  
\[
\delta_X(\varepsilon) = \inf\left\{1 - \frac{\|x_1 + x_2\|}{2} : x_1, x_2 \in S_X, \|x_1 - x_2\| = \varepsilon\right\}, \quad 0 \leq \varepsilon \leq 2.
\]  
Key properties include:\\  
(i) If \( \delta_X(1) > 0 \), then \( X \) has normal structure \cite{011}. \\ 
(ii) \( X \) is uniformly convex if and only if \( \delta_X(\varepsilon) > 0 \) for all \( \varepsilon \in (0, 2] \).\\  
(iii) \( X \) is strictly convex if and only if \( \delta_X(2) = 1 \) \cite{012}.\\  
(iv) \( X \) is uniformly convex if and only if \( \sup\{\varepsilon \in [0, 2] : \delta_X(\varepsilon) = 0\} = 0 \) \cite{011}. \\ 

In \cite{007}, the James constant \( J(X) \) is defined as:  
\[
J(X) = \sup\left\{\min\{\|x_1 + x_2\|, \|x_1 - x_2\|\} : x_1, x_2 \in S_X\right\}.
\] 
Note that
\begin{align*}
J(X) &= \sup\left\{\min\{\|x_1 + x_2\|, \|x_1 - x_2\|\} : x_1, x_2 \in B_X\right\}\\
&= \sup\left\{\|x_1 + x_2\| : x_1, x_2 \in S_X,x_1 \perp_I x_2\right\}.    
\end{align*}

In \cite{014}, authors introduced a symmetrical form of the geometric constant $L_X(t)$, established the connection between this constant and the classical function $\gamma_X(t)$, and further provided the basic properties such as the upper and lower bounds, monotonicity, and continuity of this constant. It can also effectively characterize the properties of Hilbert spaces, uniform non-squareness, and uniform normal structure.

In the geometric theory of Banach spaces, the constant introduced by G. Zbăganu is one of the important tools for studying the properties of the space, and is used to measure the deviation between the space and the inner product space. Therefore, it has attracted many scholars to conduct in-depth exploration of it.

The Zbăganu constant \cite{019} is defined as:
\[
C_Z(X) = \sup \left\{ \frac{\|x_1+x_2\| \cdot \|x_1-x_2\|}{\|x_1\|^2 + \|x_2\|^2} : x_1, x_2 \in X, \left(x_1, x_2\right) \neq(0,0) \right\}.
\]

Due to its significance, many scholars have conducted in-depth research on it. For example, Ji Gao and Satit Saejung \cite{015} generalized G. Zbăganu to a parametric form dependent on parameter $a$:
\[
C_Z(a, X) = \sup \left\{ \frac{2\|x_1 + x_2\|\|x_1 - x_3\|}{2\|x_1\|^2 + \|x_2\|^2 + \|x_3\|^2} : x_1, x_2, x_3 \in X,\ \|x_2 - x_3\| \leq a\|x_1\| \right\},
\]
where $0\leq a \leq2$.\\
They proved that if this constant satisfies the following inequality:
\[
C_Z(a, X) < (1 - a) \left( 1 + \left( \frac{1 + a}{J(a, X) + 2a} \right)^2 \right) + 2a,
 \]
then the (Banach) space possesses uniform normal structure.

Furthermore, in \cite{016}, the authors introduced the Ptolemy constant: 
\[
C_P(X) = \sup \left\{ \frac{\|x_1-x_2\| \|x_3\|}{\|x_1-x_3\| \|x_2\| + \|x_3-x_2\| \|x_1\|} : x_1,x_2,x_3 \in X \setminus \{0\},\ x_1 \neq x_2 \neq x_3 \right\},
\]
and studied the fundamental relationship between this constant and the G. Zbăganu constant.

Pal and Chandok generalized \cite{z3} the $p$-Zbăganu constant to a generalized form dependent on parameters $p, \lambda$, and $\mu$:
 
\[
{C}_Z^p(\lambda, \mu, {X}) = \sup \left\{ \frac{\|\lambda x_1 + \mu x_2\|^p \|\mu x_1 - \lambda x_2\|^p}{2^{2p-3}(\lambda^{2p} + \mu^{2p})(\|x_1\|^{2p} + \|x_2\|^{2p})} : x_1, x_2 \in {X}, \left(x_1, x_2\right) \neq(0,0) \right\},
\]
where $p\geq1, \lambda, \mu>0$ .
They proved that if this constant satisfies the condition
 \[
{C}_Z^p(\lambda, \mu, {X}) < \frac{(3\lambda - \mu)^p (\lambda + \mu)^p}{2^{2p-2}(\lambda^{2p} + \mu^{2p})},
\]
for some $\lambda, \mu>0$ ,
then the space possesses uniform non-squareness.

Alonso and Martín \cite{018} proved that $C_Z(X) \neq C_{NJ}(X)$ in general cases by constructing counterexamples, and revealed the profound connection between these two constants and the James constant $J(X)$.

In addition to the aforementioned constants, the following geometric parameter introduced in \cite{013} will also be utilized in our subsequent analysis.
\begin{definition}\cite{013}
\[
\widetilde{H}(X) = \sup\left\{ \frac{\|x\| + \|y\|}{\|x + y\|} : x, y \in X, (x, y) \neq (0, 0), x \perp_I y \right\}.
\]
\end{definition}

These concepts and constants provide the foundation for our investigation of the geometric constant \(C_Z^I(t)\) and its relationship to the structure of Banach spaces. And for papers on the Zbăganu constant, can refer to references \cite{015,018,z1,z2,z3,z4,z5}.

\section{ Main Results }

In the study of geometric properties of Banach spaces, various constants have been introduced to quantify structural characteristics such as convexity, smoothness, and orthogonality. Building upon the classical notions of orthogonality introduced by James, we now define a new geometric constant \( C_{Z}^{I}(t) \), which incorporates isosceles orthogonality and offers a refined tool for analyzing the geometry of normed spaces.

The constant  $C_Z^I(t)$  quantifies the maximal distortion in a Banach space under isosceles orthogonality by measuring the ratio between the product of two specific linear combinations of vectors and the square of their sum norm.
\begin{definition}
$$
C_Z^I(t)=\sup \left\{\frac{\left.\left(\| t x_1+(1-t) x_2\right \|\cdot \|(1-t) x_1+t x_2\right \|}{\left\|x_1+x_2\right\|^2}: x_1, x_2 \in X,\left(x_1, x_2\right) \neq(0,0), x_1 \perp_I x_2\right\},
$$
where $0\leq t \leq \frac{1}{2}$.
\end{definition}

\begin{definition}
Let $X$ be a Banach space. The function $Z_X(t):[0,1] \rightarrow[\frac{1}{2},2]$ is defined by$$
Z_X(t)=\sup \left\{\frac{\left\|x_1+t x_2\right\| \cdot\left\|x_1-t x_2\right\|}{2}: x_1, x_2 \in S_X\right\} .
$$
\end{definition}

\begin{Remark}
    Obviously, $\frac{1-t^2}{2}\leq Z_X(t) \leq\frac{(1+t)^2}{2}$, we have
    $\displaystyle \max_{t\in[0,1]}{\frac{1-t^2}{2}}=\frac{1}{2}$
    and
    $\displaystyle \min_{t\in[0,1]}{\frac{(1+t)^2}{2}}=2.$
    Hence, $\frac{1}{2}\leq Z_X(t) \leq2$.
\end{Remark}

According to the definition of $Z_X(t)$ we can easily obtain the following:
$$C_Z(X)=\sup \left\{\frac{Z_X(t)}{1+t^2}: 0 \leq t \leq 1\right\}.$$

In the following, we present the main results concerning the new geometric constant and its relationship with the function  $Z_X(t)$ , which plays a central role in characterizing the smoothness and orthogonality structure of Banach spaces. We establish the fundamental properties of this function, which will serve as a bridge to connect this constant with classical geometric notions such as uniform smoothness.

\begin{theorem}
 Let $X$ be a Banach space. Then,
$$
C_Z^I(t)=\frac{1}{2} Z_X(1-2t),
$$
where $0 \leq t \leq \frac{1}{2}$.
\end{theorem}
\begin{proof}
When $0\leq t\leq \frac{1}{2}, $ let $x_{1},x_{2}\in X,x_1 \perp_{I}x_2, $ $u_1=\frac{x_{1}+x_{2}}{2}, u_2=\frac{x_{1}-x_{2}}{2}$, then  $\|u_1\|=\|u_2\|,$ we have
$$t x_{1}+(1-t)x_{2}=u_{1}-(1-2t)u_{2},$$
$$(1-t)x_{1}+t x_{2}=u_{1}+(1-2t)u_{2},$$
let $x_1^{\prime}=\frac{u_1}{\|u_1\|}, x_2^{\prime}=\frac{u_2}{\|u_2\|}$, $x_{1}^{\prime}, x_{2}^{\prime}\in S_{X}.$ We can get the following formula,
$$
\begin{aligned}
& \frac{\left\|t x_1+(1-t) x_2\right\| \cdot\left\|(1-t) x_1+t x_2\right\|}{\left\|x_1+x_2\right\|^2}\\
&=\frac{\left\|u_1-(1-2 t) u_2\right\| \cdot\left\|u_1+(1-2 t) u_2\right\|}{2^2 \cdot\left\|u_1\right\|^2} \\
&=\frac{\left\|x_1^{\prime}-(1-2 t) x_1^{\prime}\right\|\left\|x_1^{\prime}+(1-2 t) x_2^{\prime}\right\|}{4} \\
&\leq \frac{1}{2} Z_X(1-2 t),
\end{aligned}
$$
		
\noindent so$$C_{Z}^{I}\left(t\right)\leq\frac{1}{2}Z_{X}\left(1-2t\right).$$
On the other hand, let $$v_{1}=\frac{x_{1}+x_{2}}{2}, v_{2}=\frac{x_{1}-x_{2}}{2},$$
hence $$v_{1}+v_{2}\in S_{X}, v_{1}-v_{2}\in S_{X}.$$
Then,
$$\begin{aligned}&\frac{\|x_{1}+(1-2t)x_{2}\|\cdot\|x_{1}-(1-2t)\cdot x_{1}\|}{2}\\
& =\frac{\vert\vert v_{1}+v_{2}+(1-2t)\cdot(v_{1}-v_{2})\vert\vert\cdot\vert\vert v_{1}+v_{2}-(1-2t)\cdot(v_{1}-v_{2})\vert\vert}{2\vert\vert v_{1}+v_{2}\vert\vert^{2}}\\ 
& =2\cdot\frac{||(1-t)\cdot v_{1}+t v_{2}||\cdot||t v_{1}+(1-t)v_{2}||}{||v_{1}+v_{2}||^{2}}\\  
& \leq2\cdot C_{Z}^{I}(t),
\end{aligned}$$
so
$$C_{Z}^{I}\left(t\right)=\frac{1}{2}Z_{X}\left(1-2t\right).$$
\end{proof}	

The main result of Proposition 1 is closely related to Lemma 1, which are shown in the following.
\begin{lemma}\cite{020}
A normed space \(X\) is an inner product space if and only if \(x_{1}\perp_{I}x_{2}\Longleftrightarrow x_{1}\perp_{P}x_{2}\) for all \(x_{1},x_{2}\in X\).
\end{lemma}

\begin{proposition}
    If $X$ is Hilbert space, then  $C_Z^I(t)=t-t^2$ for any $t\in[0,\frac{1}{2}]$.
\end{proposition}

\begin{proof}
If $X$ is Hilbert space, then for any \(x_1,x_2\in X\) such that \(x_1\perp_{I}x_2\), it follows that \(x_1\perp_{P}x_2\), namely \(\|x_1+x_2\|^{2}=\|x_1\|^{2}+\|x_2\|^{2}\). Therefore
\begin{align*}
\|tx_1+(1-t)x_2\| \cdot \|(1-t)x_1+tx_2\| &= t(1-t)\|x_1\|^2+t(1-t)\|x_2\|^{2}+(2t^{2}-2t+1)\langle x_1,x_2\rangle \\
&= (t-t^{2})(\|x_1\|^{2}+\|x_2\|^{2}),
\end{align*}
which implies that
\begin{align*}
\frac{\|tx_1+(1-t)x_2\| \cdot \|(1-t)x_1+tx_2\|}{\|x_1+x_2\|^{2}} &= \frac{(t-t^{2})(\|x_1\|^{2}+\|x_2\|^{2})}{\|x_1\|^{2}+\|x_2\|^{2}} \\
&= t-t^{2}.
\end{align*}
\end{proof}
        
\begin{proposition} Let $X$ be a Banach space, then $t-t^2 \leq C_Z^I(t) \leq t^2-2 t+1$.
\end{proposition} 
\begin{proof} Let $x_1=0, x_2 \neq 0$, then $x_1 \perp_I x_2$, and
$$
\begin{aligned}
& C_Z^I(t) \geq \frac{\left\|(1-t) x_2\right\|\left\|t x_2\right\|}{\left\|x_2\right\|^2} = t-t^2.
\end{aligned}
$$
On the other hand,
		
$$
\begin{aligned}
& \frac{\left.\left\|t x_1+(1-t) x_2\right\| \cdot \| (1-t\right) x_1+t x_2 \|}{\left\|x_1+x_2\right\|^2} \\
& =\frac{\left\|\frac{1}{2}\left(x_1+x_2\right)-\frac{1-2 t}{2}\left(x_1-x_2\right)\right\| \cdot\left\|\frac{1}{2}\left(x_1+x_2\right)+\frac{1-2 t}{2}\left(x_1-x_2\right)\right\|}{\left\|x_1+x_2\right\|^2} \\
& \leq \frac{(\left\|\frac{1}{2}\left(x_1+x_2\right)\|+\|\frac{1-2 t}{2}\left(x_1-x_2\right)\right\|) \cdot(\left\|\frac{1}{2}\left(x_1+x_2\right)\|+\|\frac{1-2 t}{2}\left(x_1-x_2\right)\right\|)}{\left\|x_1+x_2\right\|^2} \\
& = \frac{\left[\frac{1}{2}\left\|x_1+x_2\right\|+\frac{1-2 t}{2}\left\|x_1-x_2\right\|\right]^2}{\left\|x_1+x_2\right\|^2} \\
& =\frac{(t^2-2t+1)\|x_1+x_2\|^2}{\left\|x_1+x_2\right\|^2} \\
& =t^2-2 t+1.
\end{aligned}
$$
\end{proof}
		
In the following two examples, we demonstrate that the upper bound of $C_Z^I(t)$ is sharp.
\begin{example}Let $X=\left(\mathbb{R}^2,\|\cdot\|_1\right)$, then $C_Z^I(t)=t^2-2 t+1$.
\end{example}
\begin{proof} $\quad \operatorname{Let} x_1=(1,1), x_2=(1,-1)$, then $x_1 \perp_I x_2$, we get

$$
\left\|x_1+x_2\right\|_1=2,
$$
and
$$
\left\|t x_1+(1-t) x_2\right\|_1=\left\|(1-t) x_1+t x_2\right\|_1=2 - 2t .
$$

\noindent Hence, we have $C_Z^I(t)=t^2-2 t+1$.
\end{proof}

\begin{example}Let $X=C([a,b])$, the space of real-valued continuous functions on the interval [a,b] endowed with the supremum norm:\\
$$\|x\|=\sup_{m\in[a,b]}|x(m)|,$$
then $C_Z^I(t)=t^2-2 t+1$.
\end{example}
\begin{proof} $\quad \operatorname{Let} x_1'=\frac{1}{b-a}(m-a), x_2'=\frac{-1}{b-a}(m-a)+1\in S_{C([a,b])}$, then $x_1' \perp_I x_2'$, we have
\begin{align*}
    C_Z^I(t)&\geq \frac{\left.\left(\| t x_1'+(1-t) x_2'\right )\| \|(1-t) x_1'+t x_2'\right) \|}{\left\|x_1'+x_2'\right\|^2}\\
    &=\sup_{m\in[a,b]}\left|\frac{2t-1}{b-a}(m-a)+1-t\right|\cdot\sup_{m\in[a,b]}\left|\frac{1-2t}{b-a}(m-a)+t\right|\\
    &=t^2-2t+1.
\end{align*}
Hence, we have $C_Z^I(t)=t^2-2 t+1$.
\end{proof}

The main result of Proposition 3 is closely related to Lemma 2 which is
shown in the following.

\begin{lemma}{\cite{20}}
Let ${X}$ be a Banach space. For any $x_1, x_2 \in S_{X}$, then the following conclusions hold:

\begin{enumerate}
\item[(i)] $t \in \mathbb{R}^+ \mapsto f(t) = \| x_1 + t x_2 \|\cdot \| x_1 - t x_2 \|$ is an increasing function.
\item[(ii)] $t \in \mathbb{R}^+ \mapsto g(t) = \| t x_1 + x_2 \|\cdot \| t x_1 - x_2 \|$ is an increasing function.
\end{enumerate}
\end{lemma}

\begin{proposition}
    
Let \(X\) be a Banach space. Then

\begin{enumerate}    
\item[(i)]  $C_Z^I(t)$  is continuous on \([0,\frac{1}{2}]\). \item[(ii)]  $C_Z^I(t)$ is non-increasing on \([0,\frac{1}{2}]\).
\end{enumerate}
\end{proposition}
\begin{proof}
(i) Obvious.\\
(ii) Based on Theorem 1, we only need to prove that $Z_X(t)$ is non-decreasing function of $t$ on $[0, \frac{1}{2}]$ . 
By fixing $x_1,x_2$ on $X$ such that $x_1 \perp_I x_2$., we define that\\
$$
G(t):= \| x_1+ t x_2\| \cdot \| x_1- t x_2 \|.
$$
Applying Lemma 2, for any $t_1 \leq t_2$, we have
$Z_X(t_1) \leq Z_X(t_2).$\\
From Theorem 1 and the monotonicity of \( Z_X(t) \):
For \( t_1 < t_2 \in [0, \frac{1}{2}] \), we have \( 1 - 2t_1 > 1 - 2t_2 \).
	Since \( Z_X(t) \) is increasing, it follows that:
	\[
	Z_X(1 - 2t_1) \geq Z_X(1 - 2t_2).
	\]
	Therefore:
	\[
	C^I_Z(t_1) = \frac{1}{2} Z_X(1 - 2t_1) \geq \frac{1}{2} Z_X(1 - 2t_2) =  C^I_Z(t_2),
	\]
	as desired.
\end{proof}

\begin{theorem}
Let \(X\) be a Banach space, then \(C_Z(X)=\sup\left\{\frac{2C_Z^I\left(\frac{1-\eta}{2}\right)}{1+\eta ^{2}}:0\leq\eta\leq 1\right\}\).
\end{theorem}

\begin{proof}
First, combined with Theorem 1, we have

\[C_Z^I(t) = \frac{1}{2}\cdot Z_{X}\left(1-2t\right),\]

\noindent let \(1-2t=\eta\), then \(\eta\in[0,1]\) and

\[Z_{X}\left(\eta\right)=2\cdot C_Z^I\left(\frac{1-\eta}{2}\right).\]

\noindent Therefore, we obtain

\[C_Z(X)=\sup\left\{\frac{2C_Z^I\left(\frac{1-\eta}{2}\right)}{1+\eta^ {2}}:0\leq\eta\leq 1\right\}. \tag{1}\]

In the following two examples, we estimate the value of the geometric constant \(C_Z^I(t)\) in the \(l_{p}\) space, and use the identity of (1) to calculate the value of the von Neumann-Jordan constant in \(X=(\mathbb{R}^{2},\|\cdot\|_{1})\), which is in agreement with the known result, indicating that the identity is meaningful.
\end{proof}

\begin{example}
    
Let \(l^{p}\) (\(1<p<\infty\)) be the linear space of all sequences in \(\mathbb{R}\) for which \(\sum_{i=1}^{{\infty}}|x_{i}|^{p}<\infty\). The norm is defined as:

\[\|x\|_{p}=\left(\sum_{i=1}^{\infty}|x_{i}|^{p}\right)^{\frac{1}{p}},\]

\noindent for any sequence \(x=(x_{i})\in l^{p}\). Then

\[C_{l_{p}}^I(t)\geq 2^{-\frac{2}{p}}\left((1-t)^{p}+t^{p}\right)^{\frac{2}{p}}.\]
\end{example}

\begin{proof}
Let \(x_1=\left(\frac{1}{2^{\frac{1}{p}}},\frac{1}{2^{\frac{1}{p}}},0,\cdots,0\right),x_2= \left(\frac{1}{2^{\frac{1}{p}}},-\frac{1}{2^{\frac{1}{p}}},0,\cdots,0\right)\), then \(x,y\in S_{X}\) and

\[Z_{l_{p}}(t)\geq\frac{\|x_1+tx_2\| \cdot \|x_1-tx_2\|}{2}=\frac{1}{2} \cdot \left(\frac{(1+t)^ {p}+(1-t)^{p}}{2}\right)^{\frac{2}{p}},\]

\noindent so

\[C_{l_{p}}^I(t)\geq\frac{1}{4}\cdot\left(\frac{(2-2t)^{p}+(2t)^{p}}{2}\right )^{\frac{2}{p}}=2^{-\frac{2}{p}}\left((1-t)^{p}+t^{p}\right)^{\frac{2}{p}}.\]

\end{proof}

\begin{example}
Let \(X=(\mathbb{R}^{2},\|\cdot\|_{1})\), then \(C_Z(X)=1\).
\end{example}

\begin{proof}
In Example 1, we have

\[C_Z^I(t)=t^{2}-2t+1,\]

\noindent combined with Theorem 2, we get\\
\begin{align*}
C_Z(X)&=\sup_{\eta\in[0,1]}\left\{\frac{2C_Z^I\left(\frac{1-\eta}{2}\right)}{1+\eta ^{2}}\right\}\\
&=\sup_{\eta\in[0,1]}\left\{\frac{2\left(\left(\frac{1-\eta}{2}\right)^2-2\left(\frac{1-\eta}{2}\right)+1\right)}{1+\eta^{2}}\right\}\\
&=\sup_{\eta\in[0,1]}\left\{\frac{2\left(\frac{1-\eta}{2}-1\right)^{2}}{1+\eta^{2}}\right\}\\&=1.
\end{align*}
\end{proof}

\begin{abstract}

Below, we provides a detailed proof establishing a lower bound for the geometric constant \(C_Z^I(t)\) in the two-dimensional real space \(\mathbb{R}^2\) endowed with the \(l_p-l_q\) norm. 
\end{abstract}
\begin{example}
Let \(X = l_p-l_q\) \((1 \leq q \leq p < \infty)\) be \(\mathbb{R}^{2}\) endowed with the norm defined by:
\[
\|(x_{1}, x_{2})\| = 
\begin{cases} 
\|(x_{1}, x_{2})\|_{p}, & x_{1}x_{2} \geq 0, \\
\|(x_{1}, x_{2})\|_{q}, & x_{1}x_{2} < 0,
\end{cases}
\]
then
\begin{equation}\label{eq:main_bound}
C_{l_{p}-l_{q}}^I(t) \geq 2^{-\frac{2}{p}-2} \left[ \left(1 + 2^{\frac{1}{p}-\frac{1}{q}} - 2^{\frac{1}{p}-\frac{1}{q}+1} \cdot t \right)^{p} + \left(1 - 2^{\frac{1}{p}-\frac{1}{q}} + 2^{\frac{1}{p}-\frac{1}{q}+1} \cdot t \right)^{p} \right]^{\frac{2}{p}},
\end{equation}
where $0 \leq t \leq \frac{1}{2}$. 
\end{example}

\begin{proof}
Let \( x_1 = \left( \frac{1}{2^{\frac{1}{p}}}, \frac{1}{2^{\frac{1}{p}}} \right) \), \( x_2 = \left( \frac{1}{2^{\frac{1}{q}}}, -\frac{1}{2^{\frac{1}{q}}} \right) \), then \( x_1 \perp_{I} x_2 \), we have
\[
x_1+tx_2=\left( \frac{1}{2^{\frac{1}{p}}}+\frac{t}{2^{\frac{1}{q}}}, \frac{1}{2^{\frac{1}{p}}}-\frac{t}{2^{\frac{1}{q}}} \right) ,
\]
and
\[
x_1-tx_2=\left( \frac{1}{2^{\frac{1}{p}}}-\frac{t}{2^{\frac{1}{q}}}, \frac{1}{2^{\frac{1}{p}}}+\frac{t}{2^{\frac{1}{q}}} \right) ,
\]
we have
\[
\left(\frac{1}{2^{\frac{1}{p}}}-\frac{t}{2^{\frac{1}{q}}}\right)\cdot \left(\frac{1}{2^{\frac{1}{p}}}+\frac{t}{2^{\frac{1}{q}}}\right) \geq 0,
\]
then we obtain
\[
Z_{l_p-l_q}(t) \geq \frac{\|x + ty\|\cdot \|x - ty\|}{2} = 2^{-\frac{2}{p}-1} \left[ \left( 1 + t \cdot 2^{\frac{1}{p}-\frac{1}{q}} \right)^p + \left( 1 - t \cdot 2^{\frac{1}{p}-\frac{1}{q}} \right)^p \right]^{\frac{2}{p}},
\]
hence
\begin{align*}
C_{l_p-l_q}^I(t) &= \frac{1}{2} Z_{l_p-l_q}(1-2t) \\
&\geq 2^{-2-\frac{2}{p}} \left[ \left( 1 + 2^{\frac{1}{p}-\frac{1}{q}} - 2^{\frac{1}{p}-\frac{1}{q}+1} \cdot t \right)^p + \left( 1 - 2^{\frac{1}{p}-\frac{1}{q}} + 2^{\frac{1}{p}-\frac{1}{q}+1} \cdot t \right)^p \right]^{\frac{2}{p}}.
\end{align*}
\end{proof}

\begin{proposition}
Let  X  be a Banach space. Then 

$$
C_{Z}^{I}(t) \leq (1-t)^2 + (-2t^2+3t-1)\cdot\widetilde{H}(X) + \frac{(2t-1)^{2}}{2C_{NJ}(X)}.
$$
\end{proposition}
 
\begin{proof}
 First, note that $C_{Z}^{I}(t)$ can also be written in the following form:
\begin{equation*}
\begin{split}
C_{Z}^{I}(t) &= \sup\Bigg\{ \frac{2(\|tx_1 + (1-t)x_2\| \cdot \|(1-t)x_1 + tx_2\|)}{\|x_1 + x_2\|^{2} + \|x_1 - x_2\|^{2}} \, : \\
&\quad x_1, x_2 \in X, \, (x_1, x_2) \neq (0, 0), \, x_1 \perp_{I} x_2 \Bigg\}.
\end{split}
\end{equation*}
 
\noindent Then, we have

\begin{align*}
\|tx_1 + (1-t)x_2\| &\leq (1-t)  \|x_1 + x_2\| + (2t-1)\|x_1\|, \\
\|(1-t)x_1 - tx_2\| &\leq (1-t)  \|x_1 + x_2\| + (2t-1)\|x_2\|.
\end{align*}
\noindent Thus,
\begin{align*}
&\frac{2(\|tx_1 + (1-t)x_2\| \cdot \|(1-t)x_1 + tx_2\|)}{\|x_1 + x_2\|^{2} + \|x_1 - x_2\|^{2}} \\
&\leq \frac{2[(1-t)  \|x_1 + x_2\| + (2t-1)\|x_1\|] \cdot [(1-t)  \|x_1 + x_2\| + (2t-1)\|x_2\|]}{\|x_1 + x_2\|^{2} + \|x_1 - x_2\|^{2}} \\
&= \frac{2\left[ (1-t)^{2}\|x_1 + x_2\|^2+ (2t-1^2)\|x_1\| \cdot \|x_2\| + (2t-1)(1-t)\|x_1 + x_2\| \cdot (\|x_1\| + \|x_2\|) \right]}{\|x_1 + x_2\|^{2} + \|x_1 - x_2\|^{2}} \\
&\leq \frac{2 (1-t)^{2}\|x_1 + x_2\|^{2} + (2t-1)^{2}(\|x_1\|^{2} + \|x_2\|^{2}) + (2t-1)(1-t) \|x_1 + x_2\| \cdot 2(\|x_1\| + \|x_2\|) }{\|x_1 + x_2\|^{2} + \|x_1 - x_2\|^{2}} \\
\end{align*}
\noindent Since the geometric constant \(\widetilde{H}(X)\) also has the following equivalent definitions 
\[
\widetilde{H}(X) = \sup\left\{ \frac{2(\|x\| + \|y\|)}{\|x + y\| + \|x - y\|} : x, y \in X, (x, y) \neq (0, 0), x \perp_I y \right\}.
\]
\noindent Then, we have
\[
\frac{2(\|tx_1 + (1-t)x_2\| \cdot \|(1-t)x_1 + tx_2\|)}{\|x_1 + x_2\|^{2} + \|x_1 - x_2\|^{2}} \leq (1-t)^{2} + \frac{(2t-1)^2}{2} \cdot \frac{1}{C_{NJ}(X)} + (2t-1)(1-t) \cdot \widetilde{H}(X).
\]

\noindent Therefore,
$$
C_{Z}^{I}(t) \leq (1-t)^{2} + \frac{(2t-1)^2}{2} \cdot \frac{1}{C_{NJ}(X)} + (-2t^2+3t-1) \cdot \widetilde{H}(X).
$$
\end{proof}

\begin{proposition}
Let \(X\) be a Banach space, then
\[
\frac{1}{4}J^{2}(X)-tJ(X)-3t^{2}\leq C_Z^I(t)\leq t^2 + \frac{2t \cdot (1 - 2t)}{J(X)} + \frac{(1 - 2t)^2}{J^2(X)}.
\]

\end{proposition}

\begin{proof}
For any \(x_1,x_2\in S_{X}\), such that $x_1 \perp_{I} x_2$, since
\[
\|x_1+x_2\|=\|x_1+(1-2t)x_2+2tx_2\|\leq\|x_1+(1-2t)x_2\|+2t,
\]
and
\[
\|x_1-x_2\|=\|x_1-(1-2t)x_2-2tx_2\|\leq\|x_1-(1-2t)x_2\|+2t,
\]
so, we obtain
\begin{align*}
\min\{\|x_1+x_2\|,\|x_1-x_2\|\}^{2} 
&\leq\min\left\{\|x_1+(1-2t)x_2\|+2t,\|x_1-(1-2t)x_2\|+2t\right\}^{2} \\
&\leq \left(\|x_1+(1-2t)x_2\|+2t\right) \cdot \left(\|x_1-(1-2t)x_2\|+2t\right) \\
&= \|x_1+(1-2t)x_2\| \cdot \|x_1-(1-2t)x_2\| \\
&\quad + 2t(\|x_1+(1-2t)x_2\|+\|x_1-(1-2t)x_2\|) + 4t^{2} \\
&\leq \|x_1+(1-2t)x_2\| \cdot \|x_1-(1-2t)x_2\| \\
&\quad + 2t(\|x_1+x_2\|+2t\|x_2\|+\|x_1-x_2\|+2t\|x_2\|) + 4t^{2} \\
&\leq 2Z_{X}\left(1-2t\right)+4t^2+2t(2J(X)+4t) \\
&\leq 4C_Z^I(t)+4tJ(X)+12t^2,
\end{align*}
which implies that 
\[
 C_Z^I(t)\geq\frac{1}{4}J^{2}(X)-tJ(X)-3t^{2}.
\]

On the other hand, for any \(x_1,x_2\in S_{X}\), such that $x_1 \perp_{I} x_2$, we have
\begin{align*}
&\frac{\|tx_1 + (1 - t)x_2\| \cdot \|(1 - t)x_1 + tx_2\|}{\|x_1 + x_2\|^2} \\
&= \frac{\|t(x_1 + x_2) + (1 - 2t)x_2\| \cdot \|t(x_1 + x_2) + (1 - 2t)x_1\|}{\|x_1 + x_2\|^2} \\
&\leq \frac{(t\|x_1 + x_2\| + (1 - 2t)\|x_2\|) \cdot (t\|x_1 + x_2\| + (1 - 2t)\|x_1\|)}{\|x_1 + x_2\|^2} \\
&= \frac{t^2\|x_1 + x_2\|^2 + (1 - 2t) \cdot t \cdot \|x_1 + x_2\| \cdot (\|x_1\| + \|x_2\|) + (1 - 2t)^2 \cdot \|x_1\| \cdot \|x_2\|}{\|x_1 + x_2\|^2} \\
&\leq t^2 + \frac{2t \cdot (1 - 2t)}{J(X)} + \frac{(1 - 2t)^2}{J^2(X)}.
\end{align*}
Therefore,
\[
C_Z^I(t)\leq t^2 + \frac{2t \cdot (1 - 2t)}{J(X)} + \frac{(1 - 2t)^2}{J^2(X)}.
\]
\end{proof}

The main result of Theorem 3 is closely related to Lemma 3 which are shown in the following.
\begin{lemma}\cite{00}
Let \(X\) be a Banach space and \(x_1,x_2\in X.\) If \(x_1\perp_{I}x_2,\) then the following inequalities hold.

\begin{enumerate}
\item[(i)] \(\|x_1+\alpha x_2\|\leq|\alpha|\|x_1\pm x_2\|\) and \(\|x_1\pm x_2\|\leq\|x_1+\alpha x_2\|,\) when \(|\alpha|\geq 1.\)
\item[(ii)] \(\|x_1+\alpha x_2\|\leq\|x_1\pm x_2\|\) and \(|\alpha|\|x_1\pm x_2\|\leq\|x_1+\alpha x_2\|,\) when \(|\alpha|\leq 1.\)
\end{enumerate}
\end{lemma}

\begin{theorem}
Let \(X\) be a finite dimensional Banach space. If \(C_{Z}^I\left(t_{0}\right)=t_{0}^{2}-2t_{0}+1\) for some \(t_{0}\in\left[0,\frac{1}{2}\right]\), then \(X\) is not uniformly non-square.
\end{theorem}

\begin{proof}
Since \(C_{Z}^I\left(t_{0}\right)=t_{0}^{2}-2t_{0}+1,\) then there exist \(x_{n}\in S_{X},y_{n}\in B_{X}\) such that \(x_{n}\perp_{I}y_{n}\) and
\[
\lim_{n\rightarrow\infty}\frac{\left\|t_{0}x_{n}+\left(1-t_{0}\right)y_{n}\right\|\cdot\left\|\left(1-t_{0}\right)x_{n}+t_{0}y_{n}\right\|}{\left\|x_{n}+y_{n}\right\|^{2}}=t_{0}^{2}-2t_{0}+1.
\]

\noindent Since \(X\) is finite dimensional, then there exist \(x_{0},y_{0}\in B_{X}\) such that \(x_{0}\perp_{I}y_{0}\) and
\[
\lim_{i\rightarrow\infty}\|x_{n_{i}}\|=\|x_{0}\|\,,\,\lim_{i\rightarrow\infty}\|y_{n_{i}}\|=\|y_{0}\|\,.
\]

\noindent Combining Lemma 3, we obtain
\[
\|t_{0}x_{n}+\left(1-t_{0}\right)y_{n}\|\leq\left(1\!-\!t_{0}\right)\|x_{n}+y_{n}\|\;\;\text{and}\;\;\|(1-t_{0})x_{n}+t_{0}y_{n}\|\leq\left(1\!-\!t_{0}\right)\|x_{n}+y_{n}\|\,,
\]
thus,
\[
\frac{\left(1-t_{0}\right)^{2}\|x_{n}+y_{n}\|^{2}}{\|x_{n}+y_{n}\|^{2}}\leq t_{0}^{2}-2t_{0}+1,
\]
then 
\[
\|t_{0}x_{0}+\left(1-t_{0}\right)y_{0}\|=\left(1-t_{0}\right)\|x_{0}+y_{0}\|,
\]
and 
\[
\|(1-t_{0})x_{0}+t_{0}y_{0}\|=\left(1-t_{0}\right)\|x_{0}+y_{0}\|.
\]

On the other hand, 
\[
\|t_{0}x_{0}\!+\!(1\!-\!t_{0})y_{0}\|\leq\left(1\!-\!2t_{0}\right)\|y_{0}\| +t_{0}\,\|x_{0}\!+\!y_{0}\|,
\]
then \(\|x_{0}+y_{0}\|\leq\|y_{0}\|.\) \\
Similarly, \(\|x_{0}+y_{0}\|\leq\|x_{0}\|\) also holds. Therefore,
\[
\max\left\{\|x_{0}+y_{0}\|\,,\|x_{0}-y_{0}\|\right\}=\|x_{0}+y_{0}\|\leq \min\left\{\|x_{0}\|\,,\|y_{0}\|\right\}\leq 1<1+\delta,
\]
for any \(\delta\in\left(0,1\right)\), which shows that \(X\) is not uniformly non-square.
\end{proof}

Next, we explore a characterization of uniformly smooth spaces through the function $Z_X(t).$
\begin{proposition}
If $\displaystyle \lim_{t \to 0^{+}} \frac{Z_X(t) - \frac{1}{2}}{t} = 0,$ then $X$ is uniformly smooth.
\end{proposition}

\begin{proof}
     For any \( x_1 , x_2 \in S_X \), we obtain
\[
\frac{\|x_1 + tx_2\| \cdot \|x_1 - tx_2\|}{2}  \leq \frac{\left( \frac{\|x_1 + tx_2\| + \|x_1 - tx_2\|}{2} \right)^2}{2},
\]
then, we get
\[
Z_X(t)\leq\frac{(\rho(t)+1)^2}{2},
\]
which implies
\[
\frac{Z_X(t) - \frac{1}{2}}{t} \leq \frac{\frac{(\rho(t) + 1)^2}{2} - \frac{1}{2}}{t}  \to 0 \quad (t \to 0).
\]

\noindent Thus
\[
\displaystyle\lim_{t \to 0} \frac{\rho(t)}{t} = \lim_{t \to 0} \frac{\frac{(\rho(t) + 1)^2}{2} - \frac{1}{2}}{t} = 0,
\]
then \( X \) is uniformly smooth.
\end{proof}

\begin{theorem}
Let \( X \) be a Banach space, then \( X \) is uniformly smooth if
\[
\displaystyle \lim_{t \to \frac{1}{2}^{-}} \frac{2C_Z^I(t) - \frac{1}{2}}{1 - 2t} = 0.
\]
\end{theorem}

\begin{proof}
Let \(\alpha = 1 - 2t\), then \( t = \frac{1 - \alpha}{2} \) and \( t \to \frac{1}{2}^{-},\) that is $\alpha \to 0^+ $. We have
\[
\frac{2C_Z^I(t) - \frac{1}{2}}{1 - 2t} = \frac{Z_X (1 - 2t) - \frac{1}{2}}{1 - 2t} = \frac{Z_X (\alpha) - \frac{1}{2}}{\alpha}.
\]
Then, we have

$$\lim_{\alpha \to 0^+} \frac{Z_X(\alpha) - \frac{1}{2}}{\alpha} = 0,$$

\noindent it follows necessarily that
$$\lim_{\alpha \to 0^+} \frac{\rho_X(\alpha)}{\alpha} = 0.$$

\noindent Therefore, based on Proposition 6, \(\lim_{\alpha \to 0^+} \frac{Z_X (\alpha) - \frac{1}{2}}{\alpha} = 0\), then \( X \) is uniformly smooth.
\end{proof}

\mbox{}
\\[8pt]
{\bf Acknowledgements}\ \ Thanks to all the members of the Functional Analysis Research team of the College of Mathematics and Physics of Anqing Normal University for their discussion and correction of the difficulties and errors encountered in this paper.
\\[8pt]
{\bf Conflict of Interest}\ \ The authors declare no conflict of interest.


\begin{thebibliography}{99}
\bibitem{00}
James R.C.: Orthogonality in normed linear spaces. Duke Math. J. \textbf{12}(2), 291-302(1945)

\bibitem{001}
Brodskii, M. Milman, D. : On the center of a convex set. In Dokl. Akad. Nauk. SSSR (NS). \textbf{59}, 837-840(1948)

\bibitem{002}
Clarkson, J.A.: The von Neumann-Jordan constant for the Lebesgue space. Ann. of Math. \textbf{38}, 114-115 (1937)

\bibitem{003}
James R.C.: Uniformly non-square Banach spaces. Ann. of Math. \textbf{80}, 542-550(1964)

\bibitem{004}
Kato, M., Maligranda, L., Takahashi, Y.: On James and Jordan-von Neumann constants and the normal structure coefficient of Banach spaces. Studia Math. \textbf{144}, 275-295 (2001)

\bibitem{005}
Kirk, W.A.: A fixed point theorem for mappings which do not increase distances. Am. Math. Mon. \textbf{72}, 1004-1006 (1965)

\bibitem{006}
García-Falset, J., Llorens-Fuster, E., Mazcuñan-Navarro, E. M. : Uniformly nonsquare Banach spaces have the fixed point property for nonexpansive mappings. Journal of Functional Analysis, \textbf{233}(2), 494-514(2006)

\bibitem{007}
Gao, J., Lau, K.S. : On the geometry of spheres in normed linear spaces. J. Aust. Math. Soc. Ser. A. \textbf{48}(1), 101-112 (1990)

\bibitem{008}
Papini, P.L. : Constants and symmetries in Banach spaces. Ann. Univ. Mariae Curie-Sklodowska Sect. A (2002)

\bibitem{009}
Lindenstrauss, J.: On the modulus of smoothness and divergent series in Banach spaces. Michigan Math. J. \textbf{10}, 241-252 (1963)

\bibitem{010}
Clarkson, J.A.: Uniformly convex spaces. Trans. Amer. Math. Soc. \textbf{40}, 396-414(1936)

\bibitem{011}
Goebel, K.: Convexity of balls and fixed point theorems for mapping with nonexpansive square. Compositio Math. \textbf{22}, 269-274 (1970)

\bibitem{012}
Yang, C., Wang, F.: On estimates of the generalized Jordan-von Neumann constant of Banach spaces. JIPAM. J. Inequal. Pure Appl. Math. \textbf{7}, 1-5 (2006)

\bibitem{013}
Yuxin Wang, Qi Liu, and Mengmeng Bao. : "How orthogonality influences geometric constants?." arxiv preprint arxiv:2507.17122 (2025)

\bibitem{014}
Qichuan Ni, Qi Liu, Yuxin Wang, Jinyu Xia, Ranran Wang. Symmetric form geometric constant related to isosceles orthogonality
in Banach spaces,Filomat,(2025)

\bibitem{015}
Gao, J., Saejung, S. : Normal structure and the generalized James and Zbăganu constants. Nonlinear Analysis: Theory, Methods ; Applications, \textbf{71}(7-8), 3047-3052(2009)

\bibitem{018}
Alonso, J., Martin, P. : A counterexample for a conjecture of G. Zbăganu about the Neumann-Jordan constant. Studia Mathematica, \textbf{156}(2), 187–196(2003)

\bibitem{z1}
Yunan Cui, Meiling Zhang. : Generalized Zbăganu constant. Journal of Harbin University of Science Technology, \textbf{22}(5), (2017)

\bibitem{z2}
Zhang, J., Cui, Y. : On some geometric constants and the fixed point property for multivalued nonexpansive maps. Fixed Point Theory and Applications, \textbf{2010}(1), 596952(2010)

\bibitem{z3}
Pal, K., Chandok, S. : Generalized $p$–ZBaganu constant in Banach spaces. Mathematical Inequalities Applications, \textbf{28}(2), 307-325(2025)

\bibitem{z4}
Mizuguchi, H. : Some geometric constants and the extreme points of the unit ball of Banach spaces. Revue Roumaine De Mathematiques Pures Et Appliquees, \textbf{60}(1), 59-70(2015)

\bibitem{z5}
Q, Li., Yin, Z., Y, Wang., Liu, Q., H, Zhang. : A new constant in Banach spaces based on the Zbăganu constant $C_Z(B)$. Aims Mathematics, \textbf{10}(3), 6480-6491(2025)

\bibitem{016}
Llorens-Fuster, E., Mazcuñán-Navarro, E. M., Reich, S. : The Ptolemy and Zbăganu constants of normed spaces. Nonlinear Analysis: Theory, Methods ; Applications, \textbf{72}(11), 3984-3993(2010)

\bibitem{019}
Zbăganu, G. : An inequality of M. Radulescu and S. Radulescu which characterizes the inner product spaces. Revue Roumaine de Mathématiques Pures et Appliquées, \textbf{47}(2), 253(2002)

\bibitem{020}
Papadopoulos, A.: Surveys in Geometry I. Gewerbestrasse 11, 6330 Cham,
Switzerland: Springer Nature Switzerland AG (2022)

\bibitem{01} 
Gao, J. : On some geometric parameters in Banach spaces. J. Math. Anal. Appl. \textbf{1}, 114–122 (2007)
	
\bibitem{02}
Gao, J., Lau, K.S. : On two classes of Banach spaces with uniform normal structure. Studia Math. \textbf{99}, 41–56 (1991)
	
\bibitem{03}
Yang C, Wang F. : On a new geometric constant related to the von Neumann-Jordan constant. J. Math. Anal. Appl. \textbf{324}, 555–565 (2006)
	
\bibitem{04} 
Kato, M.,  Maligranda, L.,  Takahashi, Y. : Von Neumann–Jordan constant and some geometrical constants of Banach spaces. in “Nonlinear Analysis and Convex Analysis,” Research Institute for Mathematical Sciences 1031, pp.68–74, Kyoto University, Kyoto, Japan, (1998)
	
\bibitem{05}
Kato, M.,  Maligranda, L.,  Takahashi, Y. : On James Jordan–von Neumann constants and the normal structure coefficient of Banach spaces. Studia Math, \textbf{144}, 275–295  (2001)
	
\bibitem{06}
Alonso, J.,  Martin, P.,  Papini, P. :  Wheeling around von Neumann-Jordan constant in Banach spaces. Studia Math. \textbf{188} , 135-150 (2008)
	
\bibitem{07}
Gao, J. : Research on normal structure in a Banach space via some parameters in its dual space. Commun. Korean Math. Soc, \textbf{34}, 465–475  (2019)
	
\bibitem{08}
Jim$\acute{e}$neZ-Melado A, Llorens-Fuster E, Saejung S. : The von Neumann-Jordan constant, weak orthogonality and normal structure in Banach spaces. Proc Amer Math Soc, \textbf{134}(2), 355–364
(2006)
	
\bibitem{09}
Komuro, N.,  Saito, K.,  Tanaka, R. : On the class of Banach spaces with James constant $\sqrt2$. Math. Nachr. \textbf{289} , no. 8-9, 1005–1020  (2016)
	
\bibitem {10}
James, R. : Uniformly non-square spaces. Ann. of Math. \textbf{80},  542–550 (1964)
	
\bibitem{11}
Brodskii, M., Milman, D. : On the center of a convex Set. Dokl. Akad. Nauk SSSR(N.S.). \textbf{59} , 837-840 (1948)
	
\bibitem{12}
Prus, S. : Geometrical background of metric fixed point theory. in: W.A.Kirk, B. Sims(Eds.), Handbook of Metric Fixed Point Theory, Kluwer Acad. Publ., Dordrecht, pp. 93–132 (2001)
	
\bibitem{13}
Sims, B. : Orthogonality and fixed points of nonexpansive maps. Proc. Centre Austral. Nat. Univ. \textbf{20}, 179–186  (1988)
	
\bibitem{14}
Sims, B. : A class of spaces with weak normal structure. Bull. Austral. Math. Soc. \textbf{50}, 523–528  (1994) 
	 
\bibitem{15}
Zuo Z. : The generalized von Neumann-Jordan type constant and fixed points for multivalued nonexpansive mappings. Science Asia, \textbf{45}(3) (2019)
	 
\bibitem{16}
J. A. Clarkson. : The von Neumann-Jordan constant for the Lebesgue space. Ann. of Math.  \textbf{38}, 1, 114–115(1937)
	 
\bibitem{17}
Q. Liu., Y. Li.: On a New Geometric Constant Related to the Euler-Lagrange Type Identity in Banach Spaces. Math. \textbf{ 9}, 2, 116-128(2021)
	 
\bibitem{18}
Q. Ni, Q. Liu and Y. Zhou, Skew generalized von Neumann-Jordan constant in Banach Spaces, Hacet. J. Math. Stat. 1-9(2024)
     
\bibitem{19}
Alonso, J., Martin, P.: A counterexample to a conjecture of G. Zbăganu about the Neumann-Jordan constant. Rev. Roumaine Math. Pures Appl. \textbf{51}, 135–142(2006)

\bibitem{20}
Y, Wang., Q, Liu., H, Zhou., J, Xia., M, Toseef. : On some generalized geometric constants with two parameters in Banach spaces. arxiv preprint arxiv:2505.17600(2025)

\bibitem{23}
Benítez, C., Delrio, M.: Characterization of inner product-spaces through rectangle and square inequalities. Revue Roumaine de Mathématiques Pures et Appliquées, \textbf{29}(7), 543-546 (1984)

\end{thebibliography}
\end{document}